\documentclass[11pt, reqno]{amsart}
\usepackage{appendix}
\usepackage{xcolor}
\usepackage[utf8]{inputenc}
\usepackage[colorlinks=true,hyperindex,pagebackref,linktocpage=true]{hyperref}
\usepackage{cleveref}
\usepackage{mathtools}
\usepackage{amsfonts}
\usepackage{amssymb}
\usepackage{tikz-cd}
\usepackage[T1]{fontenc}

\hypersetup{
	colorlinks,
	linkcolor={violet!60!black},
	citecolor={cyan!60!black},
	urlcolor={orange!60!black}
}
\usepackage{stmaryrd}
\usepackage{mathrsfs}  
\usepackage{varwidth}
\theoremstyle{plain}
\newtheorem{theorem}{Theorem}[section]
\newtheorem{proposition}[theorem]{Proposition}

\newtheorem{lemma}[theorem]{Lemma}

\newtheorem{question}[theorem]{Question}

\theoremstyle{definition}
\newtheorem{definition}[theorem]{Definition}
\newtheorem{example}[theorem]{Example}
\newtheorem{remark}[theorem]{Remark}

\newcommand{\nc}{\newcommand}
\nc{\on}{\operatorname}

\nc{\Q}{\mathbb{Q}}
\nc{\Z}{\mathbb{Z}}
\nc{\cl}{\mathrm{cl}}

\nc{\fraka}{{\mathfrak a}} \nc{\bba}{{\mathbf a}}
\nc{\frakb}{{\mathfrak b}}
\nc{\frakc}{{\mathfrak c}}
\nc{\frakd}{{\mathfrak d}}
\nc{\frake}{{\mathfrak e}}
\nc{\frakf}{{\mathfrak f}}
\nc{\frakg}{{\mathfrak g}}
\nc{\frakh}{{\mathfrak h}}
\nc{\fraki}{{\mathfrak i}}
\nc{\frakj}{{\mathfrak j}}
\nc{\frakk}{{\mathfrak k}}
\nc{\frakl}{{\mathfrak l}}
\nc{\frakm}{{\mathfrak m}}
\nc{\frakn}{{\mathfrak n}}
\nc{\frako}{{\mathfrak o}}
\nc{\frakp}{{\mathfrak p}}
\nc{\frakq}{{\mathfrak q}}
\nc{\frakr}{{\mathfrak r}}
\nc{\fraks}{{\mathfrak s}}
\nc{\frakt}{{\mathfrak t}}
\nc{\fraku}{{\mathfrak u}}
\nc{\frakv}{{\mathfrak v}}
\nc{\frakw}{{\mathfrak w}}
\nc{\frakx}{{\mathfrak x}}
\nc{\fraky}{{\mathfrak y}}
\nc{\frakz}{{\mathfrak z}}
\nc{\frakA}{{\mathfrak A}}
\nc{\frakB}{{\mathfrak B}}
\nc{\frakC}{{\mathfrak C}}
\nc{\frakD}{{\mathfrak D}}
\nc{\frakE}{{\mathfrak E}}
\nc{\frakF}{{\mathfrak F}}
\nc{\frakG}{{\mathfrak G}}
\nc{\frakH}{{\mathfrak H}}
\nc{\frakI}{{\mathfrak I}}
\nc{\frakJ}{{\mathfrak J}}
\nc{\frakK}{{\mathfrak K}}
\nc{\frakL}{{\mathfrak L}}
\nc{\frakM}{{\mathfrak M}}
\nc{\frakN}{{\mathfrak N}}
\nc{\frakO}{{\mathfrak O}}
\nc{\frakP}{{\mathfrak P}}
\nc{\frakQ}{{\mathfrak Q}}
\nc{\frakR}{{\mathfrak R}}
\nc{\frakS}{{\mathfrak S}}
\nc{\frakT}{{\mathfrak T}}
\nc{\frakU}{{\mathfrak U}}
\nc{\frakV}{{\mathfrak V}}
\nc{\frakW}{{\mathfrak W}}
\nc{\frakX}{{\mathfrak X}}
\nc{\frakY}{{\mathfrak Y}}
\nc{\frakZ}{{\mathfrak Z}}
\nc{\bbA}{{\mathbb A}}
\nc{\bbB}{{\mathbb B}}
\nc{\bbC}{{\mathbb C}}
\nc{\bbD}{{\mathbb D}}
\nc{\bbE}{{\mathbb E}}
\nc{\bbF}{{\mathbb F}} \nc{\bbf}{{\mathbf f}}
\nc{\bbG}{{\mathbb G}}
\nc{\bbH}{{\mathbb H}}
\nc{\bbI}{{\mathbb I}}
\nc{\bbJ}{{\mathbb J}}
\nc{\bbK}{{\mathbb K}}
\nc{\bbL}{{\mathbb L}}
\nc{\bbM}{{\mathbb M}}
\nc{\bbN}{{\mathbb N}}
\nc{\bbO}{{\mathbb O}}
\nc{\bbP}{{\mathbb P}}
\nc{\bbQ}{{\mathbb Q}}
\nc{\bbR}{{\mathbb R}}
\nc{\bbS}{{\mathbb S}}
\nc{\bbT}{{\mathbb T}}
\nc{\bbU}{{\mathbb U}}
\nc{\bbV}{{\mathbb V}}
\nc{\bbW}{{\mathbb W}}
\nc{\bbX}{{\mathbb X}}
\nc{\bbY}{{\mathbb Y}}
\nc{\bbZ}{{\mathbb Z}}
\nc{\calA}{{\mathcal A}}
\nc{\calB}{{\mathcal B}}
\nc{\calC}{{\mathcal C}}
\nc{\calD}{{\mathcal D}}
\nc{\calE}{{\mathcal E}}
\nc{\calF}{{\mathcal F}}
\nc{\calG}{{\mathcal G}}
\nc{\calH}{{\mathcal H}}
\nc{\calI}{{\mathcal I}}
\nc{\calJ}{{\mathcal J}}
\nc{\calK}{{\mathcal K}}
\nc{\calL}{{\mathcal L}}
\nc{\calM}{{\mathcal M}}
\nc{\calN}{{\mathcal N}}
\nc{\calO}{{\mathcal O}}
\nc{\calP}{{\mathcal P}}
\nc{\calQ}{{\mathcal Q}}
\nc{\calR}{{\mathcal R}}
\nc{\calS}{{\mathcal S}}
\nc{\calT}{{\mathcal T}}
\nc{\calU}{{\mathcal U}}
\nc{\calV}{{\mathcal V}}
\nc{\calW}{{\mathcal W}}
\nc{\calX}{{\mathcal X}}
\nc{\calY}{{\mathcal Y}}
\nc{\calZ}{{\mathcal Z}}

\nc{\scrA}{{\mathscr A}}
\nc{\scrB}{{\mathscr B}}
\nc{\scrC}{{\mathscr C}}
\nc{\scrD}{{\mathscr D}}
\nc{\scrE}{{\mathscr E}}
\nc{\scrF}{{\mathscr F}}
\nc{\scrG}{{\mathscr G}}
\nc{\scrH}{{\mathscr H}}
\nc{\scrI}{{\mathscr J}}
\nc{\scrJ}{{\mathscr I}}
\nc{\scrK}{{\mathscr K}}
\nc{\scrL}{{\mathscr L}}
\nc{\scrM}{{\mathscr M}}
\nc{\scrN}{{\mathscr N}}
\nc{\scrO}{{\mathscr O}}
\nc{\scrP}{{\mathscr P}}
\nc{\scrQ}{{\mathscr Q}}
\nc{\scrR}{{\mathscr R}}
\nc{\D}{{\on{D}}}
\nc{\Div}{{\on{Div}}}
\nc{\Perv}{{\on{Perv}}}

\nc{\bnu}{{\bar{ \nu}}}
\nc{\olO}{\bar{\calO}}

\nc{\al}{{\alpha}} 
\nc{\be}{{\beta}}
\nc{\ga}{{\gamma}} \nc{\Ga}{{\Gamma}}
\nc{\hGa}{\hat{\Gamma}}
\nc{\ve}{{\varepsilon}} 
\nc{\la}{{\lambda}} \nc{\La}{{\Lambda}}
\nc{\om}{\omega} \nc{\Om}{\Omega} 
\nc{\sig}{{\sigma}} \nc{\Sig}{{\Sigma}}
\nc{\dR}{{\mathrm{dR}}}
\nc{\Perf}{{\mathrm{Perf}}}
\nc{\Gm}{{\mathbb{G}_m}}
\nc{\colim}{{\on{colim}}}
\nc{\et}{\mathrm{\acute{e}t}}

\makeatletter
\DeclareFontEncoding{LS1}{}{}
\DeclareFontSubstitution{LS1}{stix}{m}{n}
\DeclareMathAlphabet{\rhomalpha}{LS1}{stixscr}{m}{n}
\makeatother

\nc{\Spa}{\on{{Spa}}}
\nc{\Spd}{\on{{Spd}}}
\nc{\tnb}{\psi_{\rm tame}}
\nc{\oM}{\overline{{M}}}
\nc{\op}{{\on{op}}}
\nc{\ad}{{\on{ad}}}
\nc{\alg}{{\on{alg}}}
\nc{\Ad}{{\on{Ad}}}
\nc{\Adm}{{\on{Adm}}} \nc{\aff}{{\on{af}}}
\nc{\Aut}{{\on{Aut}}}
\nc{\Bun}{{\on{Bun}}}
\nc{\cha}{{\on{char}}}
\nc{\der}{{\on{der}}}
\nc{\Der}{{\on{Der}}}
\nc{\diag}{{\on{diag}}}
\nc{\End}{{\on{End}}}
\nc{\Fl}{{\calF\!\ell}}
\nc{\Tr}{{\on{Transp}}}
\nc{\TR}{{\calT\!\calR}}
\nc{\Gal}{{\on{Gal}}}
\nc{\Gr}{{\on{Gr}}}
\nc{\Hk}{{\on{Hk}}}
\nc{\rH}{{\on{H}}}
\nc{\Hom}{{\on{Hom}}}
\nc{\IC}{{\on{IC}}}
\nc{\id}{{\on{id}}}
\nc{\Id}{{\on{Id}}}
\nc{\ind}{{\on{ind}}}
\nc{\Ind}{{\on{Ind}}}
\nc{\Lie}{{\on{Lie}}}
\nc{\Pic}{{\on{Pic}}}
\nc{\pr}{{\on{pr}}}
\nc{\Res}{{\on{Res}}}
\nc{\res}{{\on{res}}} \nc{\Sat}{{\on{Sat}}}
\nc{\spc}{{\on{sc}}}
\nc{\drv}{{\on{der}}}
\nc{\sgn}{{\on{sgn}}}
\nc{\Spec}{{\on{Spec}}}\nc{\Spf}{\on{Spf}} 
\nc{\Sph}{\on{Sph}}
\nc{\St}{{\on{St}}}
\nc{\tr}{{\on{tr}}}
\nc{\Mod}{{\mathrm{-Mod}}}
\nc{\Hilb}{{\on{Hilb}}} 
\nc{\Ext}{{\on{Ext}}} 
\nc{\vs}{{\on{Vec}}}
\nc{\ev}{{\on{ev}}}
\nc{\nO}{{\breve{\calO}}}
\nc{\tS}{{\tilde{S}}}
\nc{\spe}{{\on{sp}}}
\nc{\loc}{{\on{loc}}}
\nc{\pre}{{\on{pre}}}

\nc{\dimt}{{\on{dim.trg}}}

\nc{\co}{\colon}
\nc{\dia}{{\diamondsuit}}

\nc{\nscrR}{{\mathscr{R}^{\on{nr}}}}

\nc{\GL}{{\on{GL}}}

\nc{\Gl}{\on{Gl}} 
\nc{\GSp}{{\on{GSp}}}
\nc{\gl}{{\frakg\frakl}}
\nc{\SL}{{\on{SL}}} 
\nc{\SU}{{\on{SU}}} 
\nc{\SO}{{\on{SO}}}
\nc{\PGL}{{\on{PGL}}}

\nc{\Conv}{{\on{Conv}}}
\nc{\Rep}{{\on{Rep}}}
\nc{\Dom}{{\on{Dom}}}
\nc{\red}{{\on{red}}}
\nc{\act}{{\on{act}}}
\nc{\nr}{{\on{nr}}}
\nc{\ctf}{{\on{ctf}}}

\nc{\str}{{\on{-}}} 
\nc{\os}{{\bar{s}}}
\nc{\oeta}{{\bar{\eta}}}

\nc{\hookto}{\hookrightarrow}
\nc{\longto}{\longrightarrow}
\nc{\leftto}{\leftarrow}
\nc{\onto}{\twoheadrightarrow}
\nc{\lonto}{\twoheadleftarrow}

\nc{\pot}[1]{ [\hspace{-0,5mm}[ {#1} ]\hspace{-0,5mm}] }
\nc{\rpot}[1]{ (\hspace{-0,7mm}( {#1} )\hspace{-0,7mm}) }
\nc{\smallpot}{{ <\hspace{-1,0mm}<}}

\numberwithin{equation}{section}

\begin{document}
	
	\title{Perfectoid Nullstellensatz: Results and counterexamples.}
	
	\author[I. Gleason]{Ian Gleason}
	\address{Mathematisches Institut der Universit\"at Bonn, Endenicher Allee 60, Bonn, Germany}
	\email{igleason@uni-bonn.de}
	
	\begin{abstract}
		We give necessary conditions and we give sufficient conditions for perfectoid Nullstellensatz to hold. As a consequence, we prove that perfectoid Nullstellensatz does not hold for $\bbC_p$ and other natural $p$-adic fields.  
	\end{abstract}

	\maketitle
	\tableofcontents
	
	\section{Introduction}
	A weak version of Hilbert's Nullstellensatz says that every finite collection \[S:=\{p_1,\dots,p_m\}\subseteq \bbC[x_1,\dots,x_n]\] of polynomial functions in $n$-variables over the complex numbers fall into two cases: either all elements of $S$ have a common solution in $\bbC^n$, or one can find a set of polynomials $q_1,\dots,q_m$ such that \[1=\sum_{i=1}^m q_ip_i.\] 
	
	Recall that for a tuple $\bar{y}\in \bbC^n$ we have an evaluation map 
	\begin{align}
		\on{ev}_{\bar{y}}:\bbC[x_1,\dots,x_n]\to \bbC \\
		\on{ev}_{\bar{y}}(f)=f(\bar{y}).
	\end{align}
	Let $I_{\bar{y}}:=\on{ker}(\on{ev}_{\bar{y}})$, we call ideals obtained in this way \textit{evaluation ideals}.
	One can reformulate the statement to say that every maximal ideal of $\bbC[x_1,\dots,x_n]$ is an evaluation ideal. 
	With this reformulation in mind we make the following definitions.
	Fix a field $C$ and a $C$-algebra $R$. 
	\begin{definition}
		 Let $I \subseteq R$ be an ideal. We say that $I$ is an \textit{evaluation ideal} if the composition map $C\to  R/I$ is an isomorphism. 
	\end{definition}
	\begin{definition}
		We say that $(C,R)$ satisfies the Nullstellensatz theorem, or that $(C,R)$ is a \textit{Nullstellensatz pair} if every maximal ideal of $R$ is an evaluation ideal. 
	\end{definition}
	The following is another reformulation, see \cite[Example 1.4.4, 1.4.5]{Hart}.
	\begin{theorem}
		\label{algebraicnull}
		The pair $(C,C[x_1,\dots,x_n])$ is a Nullstellensatz pair if and only if $C$ is algebraically closed.	
	\end{theorem}
	Now suppose that $C$ is a complete non-Archimedean analytic field. We denote by $O_C\subseteq C$ its valuation ring, by $|\cdot|:C\to \bbR^+\cup \{0\}$ its norm, by $\varpi\in O_C$ a choice of pseudouniformizer, by $\frakm\subseteq O_C$ the unique maximal ideal and by $k:=O_C/\frakm$ the residue field of $O_C$. 
	Let $T_n=C\langle x_1,\dots,x_n\rangle$ denote the Tate algebra of convergent power series in $n$-variables. 
	The following is the version of the Nullstellensatz in the context of rigid geometry, see \cite[Corollary 11, Corollary 12 \S 2.2, Theorem 4, Corollary 6 \S 3.2]{Bosch}.
	\begin{theorem}
		\label{rigidnull}
		The pair $(C,T_n)$ is a Nullstellensatz pair if and only if $C$ is algebraically closed.	
	\end{theorem}

	Now we fix a prime number $p$ and suppose that $C$ is a perfectoid field of residue characteristic $p$ \cite[Definition 3.6]{Sch17}. Let $R_n=C\langle x_1^{{1}/_{p^\infty}},\dots,x_n^{{1}/_{p^\infty}}\rangle$ denote the perfectoid Tate algebra in $n$-variables. 
	This algebra is obtained from $O_C[x_1^{{1}/_{p^\infty}},\dots x_n^{{1}/_{p^\infty}}]$ by taking its $\varpi$-adic completion and inverting $\varpi$. 
	We are concerned with the following question: 
	\begin{question}
		\label{mainquestion}
		What conditions are sufficient and what conditions are necessary for $(C,R_n)$ to be a Nullstellensatz pair? 		
	\end{question}
	Evidently for $(C,R_n)$ to be a Nullstellensatz pair it is necessary that $C$ be algebraically closed, and one may be tempted by \Cref{algebraicnull} and \Cref{rigidnull} to believe that this condition is also sufficient. 
	Unfortunately, this is not the case. 
	\begin{theorem}
		\label{notype3point}
		Suppose that $n\geq 2$ and that $|C^\times|\subsetneq \bbR^+$ is a proper subset. Then $(C,R_n)$ is \textbf{not} a Nullstellensatz pair. 
	\end{theorem}
	\Cref{notype3point} gives, outside of the case $n=1$, an additional necessary condition for perfectoid Nullstellensatz to hold. Namely, that the value group of its norm map must be as large as possible. 
	The following example shows that perfectoid Nullstellensatz fails for the fields that show up in nature. 
	\begin{example}
Let $\bbC_p$ denote the $p$-adic completion of an algebraic closure of $\bbQ_p$.
The pair $(\bbC_p,\bbC_p\langle x_1^{{1}/_{p^\infty}},\dots,x_n^{{1}/_{p^\infty}}\rangle)$ is a Nullstellensatz pair if and only if $n=1$. 
Indeed, $|\bbC_p^\times|= p^\bbQ\subsetneq \bbR^+$. 
One can argue in an analogous way for the completions of the algebraic closures of fields that are topologically of finite transcendence degree over $\bbQ_p$. 
	\end{example}
	We now give positive results, the first one says that perfectoid Nullstellensatz holds in dimension $1$. 
	\begin{proposition}
		\label{1dimeNullstellensatz}
		The pair $(C,C\langle x^{{1}/_{p^\infty}}\rangle)$ is a Nullstellensatz pair if and only if $C$ is algebraically closed.
	\end{proposition}
	Our second result says that perfectoid Nullstellensatz holds as long as one works with fields that are large enough and suitably complete. 
	This might be useful in a situation where one is allowed to work v-locally.
	\begin{theorem}
		\label{bigenoughfield}
	Suppose that all of the following conditions hold:
	\begin{enumerate}
		\item $C$ is algebraically closed. \label{condalgebclosed}
		\item $|C^\times|=\bbR^+$. \label{condtype3}
		\item $C$ is spherically complete. \label{condspheric}
		\item $k$ is uncountable. \label{conduncount}
	\end{enumerate}
	then $(C,R_n)$ is a Nullstellensatz pair.
	\end{theorem}
	\begin{remark}
		As explained in \Cref{notype3point} conditions \ref{condalgebclosed} and \ref{condtype3} are necessary as long as $n\geq 2$. We suspect that condition \ref{condspheric} is also necessary, but that condition \ref{conduncount} is not. Nevertheless, we do not know this.  	
	\end{remark}
	Let us describe briefly our methods. One can first reduce the proofs of \Cref{notype3point}, \Cref{1dimeNullstellensatz} and \Cref{bigenoughfield} to the case in which $C$ is a characteristic $p$ field by using the tilting equivalence. 
	In characteristic $p$, the algebra $R_n$ represents the unit ball $\bbB^n_C$ in the category of diamonds, and to every maximal ideal of $R_n$ one can associate a unique point in $\bbB^n_C$.
	We use the Berkovich classification of points in $\bbB^1_C$ for our arguments.
	
	For proving \Cref{bigenoughfield} we use the assumptions on $C$ to conclude that the Berkovich unit ball only has Type I points and Type II points.
	Type I points correspond to evaluation ideals and when $x\in \bbB^n_C$ projects to a Type II point in $\bbB^1_C$ then $k(x)=O_{C(x)}/\frakm$ is a transcendental extension of $k$. 
	We exploit the fact that $k(x)$ has uncountable dimension as a $k$-vector space whereas ``the elements coming from $R_n$'' are countably dimensional.

	The proof of \Cref{notype3point} is through an explicit computation.
	We consider a Type III point in $y\in \bbB^1_C$ corresponding to an element $r\in \bbR^+\setminus |C^\times|$. 
	This point induces a (perfected) residue field $C(y)$ which is a nontrivial extension of $C$, and we construct a surjection $R_3\to C(y)$ exhibiting the failure of perfectoid Nullstellensatz, when $n=3$. 
	This construction resembles a division algorithm.
	The cases $n=2$ and $n\geq 4$ can be easily reduced to the case $n=3$.
	As a byproduct of the failure of Nullstellensatz, we can answer negatively a question of Hansen \cite{HAN21}.
	\begin{example}
		\label{example:David}
		Let $y\in \Spa \bbC_p\langle x_1^{\frac{1}{p^\infty}}, x_2^{\frac{1}{p^\infty}}, x_3^{\frac{1}{p^\infty}}\rangle$ correspond to a Type III point.
		The map constructed in the proof of \Cref{mainnotype3point} \[Z:=\Spa C(y)\to \Spa \bbC_p\langle x_1^{\frac{1}{p^\infty}}, x_2^{\frac{1}{p^\infty}}, x_3^{\frac{1}{p^\infty}}\rangle,\] is a Zariski closed immersion with Krull dimension $0$ such that $H^1_{\acute{e}t}(Z,\bbF_\ell)\neq 0$.
		This contradicts the conjectural statement \cite[Conjecture 1.11]{HAN21}.
		Indeed, $H^1_{\acute{e}t}(Z,\bbF_\ell)=H^1_{\acute{e}t}(\Spec\, C(y),\bbF_\ell)$ and we may use the Kummer sequence to identify this set with $C(y)^\times/(C(y)^\times)^\ell$. 
		Now, $C(y)^\times/(C(y)^\times)^\ell\neq 0$ since it admits a surjective map to $|C(y)^\times|/|(C(y)^\times)|^\ell=p^\bbQ\oplus r^{\bbZ[\frac{1}{p}]}/ p^\bbQ\oplus r^{\ell\cdot \bbZ[\frac{1}{p}]}=\bbF_\ell$.
	\end{example}
	\subsection{Acknowledgements}
	This paper was written during stays at Max-Planck-Institut für Mathematik and Universität Bonn, we are thankful for the hospitality of these institutions. The project has received funding by DFG via the Leibniz-Preis of Peter Scholze.

	We would like to thank Peter Scholze, Linus Hamann, and Mingjia Zhang for looking at an early draft and interesting conversations on the work. We thank David Hansen for suggesting \Cref{example:David}.
	We thank the anonymous referees for a very detailed and helpful report. 
	\section{Preparations}
	\subsection{Zariski closed subsets, tilting and points in the unit ball.}
		Let $X=\Spa(R,R^+)$ be an affinoid perfectoid space.	
		\begin{definition}(\cite[Definition 5.7]{Sch17})
		A closed subset $Z\subseteq X$ is said to be a \textit{Zariski closed} set if there is an ideal $I\subseteq R$ such that $Z$ is the vanishing locus of $I$. 
	\end{definition}
	By \cite[Lemma II.2.2]{Sch15}, a Zariski closed subset of $X$ is always represented by an affinoid perfectoid space $Z=\Spa(S,S^+)$.
	Moreover, by \cite[Theorem 5.8]{Sch17} the natural map $R\to S$ is always surjective.

	Recall that tilting induces an equivalence between perfectoid spaces over $X$ and perfectoid spaces over its tilt $X^\flat$ \cite[Corollary 3.20]{Sch17}. 
	By \cite[Lemma II.2.7]{Sch15} and \cite[Theorem 5.8]{Sch17} a map of affinoid perfectoid spaces $Z\to X$ is a Zariski closed immersion if and only if the map of tilts $Z^\flat \to X^\flat$ is. 
	This has the following easy consequence.
	\begin{proposition}
		\label{workcharp}
		$(C,R_n)$ is a Nullstellensatz pair if and only if $(C^\flat, R_n^\flat)$ is a Nullstellensatz pair.
	\end{proposition}
	\begin{proof}
		The tilting equivalence gives a one-to-one correspondence between maximal ideals of $R_n$ and maximal ideals of $R_n^\flat$.
		More precisely, if $\frakm\subseteq R_n$ is a maximal ideal with corresponding residue field $C_\frakm$, then the corresponding ideal $\frakm^\flat\subseteq R_n^\flat$ is the kernel of the surjection $R_n^\flat\to C_\frakm^\flat$. 
		Furthermore, $\frakm$ is an evaluation ideal if and only if $\frakm^\flat$ is an evaluation ideal. 
		Indeed, the map $C\to C_\frakm$ is an isomorphism if and only if $C^\flat\to C_\frakm^\flat$ is . 
	\end{proof}
	We will use \Cref{workcharp} to reduce the proof of the main theorems to the case where $C$ is a characteristic $p$ field. 
	In characteristic $p$, the algebra $R_n$ represents $\bbB_C^n$, the unit ball in $n$-variables.
	Note that $\Spa C\langle x_1,\dots,x_n\rangle\cong \Spa C\langle x_1^{{1}/_{p^\infty}},\dots x^{{1}/_{p^\infty}}_n\rangle$ as topological spaces. For this reason, the notion of closed subset agrees for both adic spaces, but the notion of Zariski closed subset in one and the other are very different. 

	\subsection{Spherically complete fields}
	Recall that a field extension of non-archimedean valued fields $F$ over $C$ is said to be an \textit{immediate extension} \cite[Definition 6.9]{Ultramet} if $|F^\times|=|C^\times|$ and the map of residue fields $O_C/\frakm\to O_F/\frakm_F$ is an isomorphism.
	A non-archimedean valued field is said to be \textit{maximally complete} if it has no proper immediate extensions. 
	By \cite[Theorem 6.12]{Ultramet} maximally complete fields coincide with spherically complete fields, and by \cite[Theorem 6.13]{Ultramet} any non-archimedean valued field admits a maximally complete immediate extension, which we will refer to as a \textit{spherical completion}.

	In general, a spherical completion of a field $C$ might be hard to describe and two different spherical completions might not even be isomorphic \cite[Remark 6.21]{Ultramet}.
	Nevertheless, every non-archimedean field admits a valuation preserving injection into a field of a very simple form.

	Let $\Gamma\subseteq \bbR$ be an ordered abelian subgroup, and let $k$ be field. 
	We let $k\rpot{\Gamma}$ denote the set of pairs $\{(\calP,g)\}$ where $\calP\subseteq \Gamma$ is a well-ordered subset and $g:\calP\to k^\times$ is a function. 
	Alternatively, one can think of $k\rpot{\Gamma}$ as the set of functions $g:\bbR\to k$ whose support is well-ordered and contained in $\Gamma$.
	The addition rule is given by point-wise addition $[g_1+g_2](\gamma)=g_1(\gamma)+g_2(\gamma)$ and multiplication is given by convolution $[g_1\cdot g_2](\gamma)=\sum \limits_{\gamma_1\in \Gamma} g_1(\gamma_1)g_2(\gamma-\gamma_1)$. 
	We can endow this ring with a valuation by letting $|g|_{k\rpot{\Gamma}}=e^{-\gamma_g}\in \bbR$ where $\gamma_g$ is the smallest element in the support of $g$ and $|g|_{k\rpot{\Gamma}}=0$ when $g=0$. 
	\begin{theorem}(\cite[Theorem 3.16]{Ultramet})
		For all $k$ and $\Gamma\subseteq \bbR$, $k\rpot{\Gamma}$ is a complete non-archimedean valued field with residue field isomorphic to $k$ and value group $\Gamma$.
	\end{theorem}
	\begin{definition}
		Any field of the form $k\rpot{\Gamma}$ is called a \textit{Hahn field}.	
	\end{definition}
	\begin{theorem}(\cite[Theorem 7.3, Corollary 7.4]{Ultramet})
		\label{embeddonHahn}
		Let $C$ be a complete non-Archimedean field, let $k$ be its residue field and let $\Gamma=\on{log}(|C^\times|)\subseteq \bbR$. 
		Suppose that $\on{char}(C)=\on{char}(k)$, that $k$ is algebraically closed, and that $\Gamma$ is divisible. 
		The following hold:
		\begin{enumerate}
			\item If $C$ is spherically complete, then $C$ is isomorphic to $k\rpot{\Gamma}$. 
			\item There is an immediate extension of $C$ that is isomorphic to $k\rpot{\Gamma}$.  
		\end{enumerate}
	\end{theorem}
	A consequence of \Cref{embeddonHahn} is that every non-Archimedean field $C$ of equal characteristic may be embedded into a Hahn field $\frakC$ in such a way that the residue field of $\frakC$ is $k$ again.
	Indeed, immediate extensions, by definition, preserve the residue field.
	We will use this observation in the proof of \Cref{mainnotype3point}. 
	\begin{remark}
	A mixed characteristic analogue of \Cref{embeddonHahn} can be found in \cite[Corollary 7.17]{Ultramet}. 
	The role of Hahn fields is taken by the so called $p$-adic Mal'cev--Neumann fields.
	We will not need to study these fields since all of our arguments reduce by tilting to the equal characteristic case.
	\end{remark}
	\section{Results}
	In this section we give the proof of \Cref{bigenoughfield}. 

	\begin{theorem}
		\label{maintextbigenoughfield}
	Suppose that all of the following conditions hold:
	\begin{enumerate}
		\item $C$ is algebraically closed. \label{maincondalgebclosed}
		\item $|C^\times|=\bbR^+$. \label{maincondtype3}
		\item $k$ is uncountable. \label{mainconduncount}
		\item $C$ is spherically complete. \label{maincondspheric}
	\end{enumerate}
	then $(C,R_n)$ is a Nullstellensatz pair.
	\end{theorem}
	\begin{proof}
		By \Cref{workcharp}, it suffices to show that $(C^\flat,R_n^\flat)$ is a Nullstellensatz pair. 
		Now, we claim that $C^\flat$ also satisfies the conditions above.
		Indeed, that the first three conditions follow from the tilting compatibilities: $|C^\flat|=|C|$, $k=O_C/C^{\circ \circ}=O_{C^\flat}/C^{\flat,\circ\circ}$ and $(\Spec\, C)_{\acute{e}t}\cong  (\Spec\, C^\flat)_{\acute{e}t}$.
		For the last condition, let $K^\flat/C^\flat$ be an immediate extension which we may assume to be algebraically closed.
		By the tilting equivalence, there is an algebraically closed field extension $F/C$ with $F^\flat=K$.
		Now, $|F|=|K|=\bbR$ and the residue field of $F$ is also $k$ (since this holds for $K=F^\flat$), this implies that $F$ is an immediate extension of $C$ and since $C$ is spherically complete $C\cong F$ which implies $C^\flat\cong K$ as we wanted to show.

		By the above, we may assume without loss of generality that $C$ is of characteristic $p$.
		Let $x\in \bbB^n_C$ denote the point associated to a maximal ideal $I\subseteq R_n$, let $C(x)=R_n/I$ be the residue field at $x$, and assume for the sake of contradiction that $C(x)\neq C$. 
		By induction, we may also choose $n$ to be minimal for which such an ideal $I$ exists. 
		Let $x_1$ be the image of $x$ under the first projection map $\pi_1:\bbB_C^n\to \bbB_C^1$. 
		Recall the Berkovich classification of points in the unit ball \cite[\S 1.4.4]{Ber90}.
		Since $|C^\times|=\bbR^+$ and $C$ is spherically complete, $\bbB^1_C$ does not have Type III or Type IV points. 
		By minimality of $n$, $x_1$ is also not a Type I point, so $x_1$ must be a Type II point. 
		Indeed, if $x_1$ was a Type I point then then the residue field at $x_1$ must be $C$ since $C$ is algebraically closed. 
		Moreover, the fiber $\pi_1^{-1}(x_1)$ is $\bbB_C^{n-1}$ and we can write $C(x)=R_{n-1}/I$ contradicting the minimality of $n$. 
		This implies that the residue field $k(x)$ of $C(x)$ is a transcendental extension of $k$. 
		
		Let $f:R_n\to C(x)$ denote the quotient map, let $R'_n=f^{-1}(O_{C(x)})\subseteq R_n$, and let $A_n\subseteq k(x)$ the image of $R'_n$ in $k(x)$. 
		We claim that $A_n$ is a proper subring of $k(x)$, which implies that $f:R_n\to C(x)$ is not surjective. 
		This gives the contradiction. 
		
		By \Cref{embeddonHahn} we may embed $C(x)$ into a Hahn field ${\frakC(x)}$. 
		We may interpret $x\in \bbB^n_C$ as a map $f:R_n\to \frakC(x)$ with $f(R_n)=C(x)$. 
		One can describe $\frakC(x)$ explicitly, it is of the form $\overline{k(x)}\rpot{\bbR}$, and maps $R_n\to \frakC(x)$ are determined by the choice of a tuple $(c_1,\dots,c_n)\in O^n_{\frakC(x)}$. 
		We may interpret each element $c_i$ as a pair $=(\calP_{c_i},g_{c_i})$ where $\calP_{c_i}\subseteq \bbR$ is a well-ordered subset and $g_{c_i}:\bbR\to \overline{k(x)}^\times$ is a function whose support is $\calP_{c_i}$. 
		Similarly, for $q\in R_n$ we may interpret $f(q)$ as a tuple $(\calP_q,g_q)$, and $q\in R'_n$ if and only if the smallest element of $\calP_q$ is larger or equal to $0$.
		Moreover, if $q\in R'_n$ its image in $A_n$ is $g_q(0)$.

		Let $S\subseteq \overline{k(x)}$ denote the set of elements of the form $s=g_{c_i}(r_i)$ where $r_i\in \calP_{c_i}$ for some $i\in \{1,\dots,n\}$.
		It is not hard to see that 
		\[A_n\subseteq k[S^{\frac{1}{p^\infty}}]\cap k(x)\subseteq \overline{k(x)}.\] 
		Furthermore, note that $S$ is countable. 
		Indeed, for $i\in \{1,\dots,n\}$ the set $\calP_i\subseteq \bbR$ is countable since it is a well-ordered subset of the real numbers. 
		This readily implies that $k[S^{\frac{1}{p^\infty}}]$ has a countable basis as a $k$-vector space, which implies the same of $A_n$.

		We finish by observing that $k(x)$ has uncountable dimension over $k$.
	Indeed, if $t\in k(x)$ is a transcendental element, then the set $\{\frac{1}{t-c}\mid c\in k\}$ is linearly independent over $k$. 
	By our assumption \ref{mainconduncount} this set is uncountable.
	\end{proof}
	In the case of one variable, the perfectoid Nullstellensatz holds more generally.
	\begin{proposition}
		The pair $(C,R_1)$ is a Nullstellensatz pair if and only if $C$ is algebraically closed.		
	\end{proposition}
	\begin{proof}
		Necessity is evident. 
		Suppose $C$ is algebraically closed and of characteristic $p$. 
		We use the classification of points in the unit ball $\bbB^1_C$. 
		Take $I\subseteq C\langle x^{{1}/_{p^\infty}}\rangle$ a maximal ideal and let $f\in I$ be a non-zero element. 
		Let $Z_f$ denote the zero locus of $f$, we claim that $Z_f=\underline{S}\times \Spa(C)$ for a profinite set $S$. 
		Equivalently, we claim that $Z_f\to \Spa(C)$ is quasi-pro-\'etale.
		By \cite[Proposition 10.11.(v)]{Sch17}, it suffices to prove this after a v-cover of $\Spa(C)$, so we may enlarge $C$ to assume that $|C^\times|=\bbR^+$ and that $C$ is spherically complete.
		
		Let $z\in \bbB_C^1$, we prove that $z\notin Z_f$ whenever $z$ is of Type II. 
		By changing coordinates we may always assume that $z$ corresponds to the ball of radius $r\leq 1$ centered at the origin. 
		Recall that $C\langle x^{{1}/_{p^\infty}}\rangle$ is the completed colimit of the system of rings:
		\begin{equation}
			C\langle x\rangle \xrightarrow{x \mapsto x^p} 	C\langle x\rangle \xrightarrow{x \mapsto x^p} \dots	
		\end{equation}
		Write $f$ in the form 
\begin{equation}
		f=\sum_{m\in \bbZ[\frac{1}{p}]} a_{m}x^{m}
	\end{equation} with $a_{m}\in C$ such that for all $\varepsilon \in \bbR^+$ only finitely many terms in the sum have value $|a_m|>\varepsilon$.
		Choose a monomial of the form $a_{M}x^{M}$ appearing in $f$, and let $\epsilon=|a_{M}|\cdot r^{M}$.
		We can approximate $f$ by the finite sum of terms $f_\epsilon=\sum \limits_{m\in S_\epsilon} a_{m}x^{m}$ with $S_\epsilon:= \{m\in \bbZ[\frac{1}{p}] \mid |a_m|\geq \epsilon\}$. 
		By construction, $|f-f_\epsilon|_z<\epsilon$.
		Moreover, since $a_M x^M$ shows up as a monomial of $f_\epsilon$ we also know that $|f_\epsilon|_z=\on{sup} \limits_{m\in S_\epsilon}|a_{m}|r^{m}\geq |a_M|r^M= \epsilon$ from which it follows that $|f|_z\geq \epsilon$. 

		Now that we know $Z_f$ consists only of Type I points, we claim that the natural map $Z_f\to \underline{\pi_0(Z_f)}\times \Spa(C)$, is an isomorphism. Since both spaces are proper over $\Spa(C)$, by \cite[Lemma 11.11]{Sch17} it suffices to prove that every component of $Z_f$ consists of one point. But if two Type I points lie in the same connected component of $Z_f$, the whole Berkovich path between them in $\bbB^1_C$ would also lie in $Z_f$. 
		This contradicts the fact that $Z_f$ does not contain Type II points. 
		From the above it follows easily that $(C,R_n)$ is a Nullsellensatz pair. Indeed, the map $\Spa R_n/I \to Z_f\to \Spa C$ must factor through a connected component of $y\in \pi_0(Z_f)$, which has residue field $C$. 
	\end{proof}

	\section{The Counterexample}
	We now discuss the proof of \Cref{notype3point}. 
	In other words, we give a counterexample to perfectoid Nullstellensatz for every field $C$ such that $|C^\times|\neq \bbR^+$. 
	The counterexample relies on the explicit description of the residue field of a Type III point in the perfectoid unit ball, \Cref{ExplicitTypeIII}.
	\begin{theorem}
		\label{mainnotype3point}
		Suppose that $n\geq 2$ and that $|C^\times|\subsetneq \bbR^+$ is a proper subset. Then $(C,R_n)$ is \textbf{not} a Nullstellensatz pair. 
	\end{theorem}
	\begin{proof}
		Without losing generality we may assume that $C$ is algebraically closed and by \Cref{workcharp} that $C$ is of equal characteristic $p$.
		Assume that $n\geq 3$, and that $|C^\times|\neq \bbR^+$.
		With this setup, we construct a non-trivial field extension $C(y)$ over $C$ and a surjective map $f_y:R_3\to C(y)$.
		The kernel of $f_y$ is a maximal ideal that is not an evaluation ideal.
		Since for every $n\geq 3$ we have a surjective map $R_n\to R_3$, this proves that $(C,R_n)$ is not a Nullstellensatz pair whenever $n\geq 3$. 
		We deduce the case $n=2$ from the case $n=3$ in \Cref{casen=2} below.

		Let $r\in \bbR^+\setminus |C^\times|$ with $r<1$, and fix a pseudouniformizer $\varpi\in C$. 
		We can consider the unit ball of radius $r$ centered at the origin and this gives rise to a Type III point $y\in \bbB^1_C$. 
		We let $C(y)$ denote the (perfected) residue field at this point. 
		By \Cref{ExplicitTypeIII}, we can describe this field explicitly.	
		It consists of power series expressions 
		\begin{equation}
			\beta=\sum \limits_{q\in \bbZ[\frac{1}{p}]} b_q x^q
		\end{equation}
		where $b_{q}\in C$ and subject to the constraint that for any $\varepsilon >0$ the set $\{q\in \bbZ[\frac{1}{p}]\mid |b_{q}|r^{q}>\varepsilon\}$ is finite.  
		Addition and multiplications are given by the evident formulas and the value $|\beta|_y$ is given by $\sup \limits_{q\in \bbZ[\frac{1}{p}]} |b_{q}|r^{q}$. 
		Moreover, the supremum is always a maximum and it is uniquely attained. 
		We call this term the \textit{leading monomial}. 
	
	We divide the proof in two steps. The first step, corresponds to \Cref{lema-adapted}, it gives a criterion for when a map $R_3\to C(y)$ is surjective. The second step, corresponding to \Cref{lema-construction}, gives an explicit construction that satisfies the criterion of \Cref{lema-adapted}. 

	\begin{definition}
		Let $0<s<1$ and let $q\in \bbZ[\frac{1}{p}]$. We say that an element $\beta\in C(y)$ is $(q,s)$-adapted if the following conditions hold:  	
		\begin{enumerate}
			\item $s<|\beta|_y\leq 1$.
			\item If $\beta=\sum \limits_{j\in \bbZ[\frac{1}{p}]} b_j x^j$ then $|b_q|r^q=|\beta|_y$. In other words, $b_qx^q$ is the leading monomial.
			\item $|\beta-b_qx^q|_y\leq s\cdot |\varpi|$.
		\end{enumerate}
	\end{definition}

		\begin{lemma}
			\label{lema-adapted}
			Let $f:R_3\to C(y)$ be a continuous map. Suppose there exists $0<s<1$ such that for any $q\in \bbZ[\frac{1}{p}]$, there is an element $\alpha\in O_C\langle x^{{1}/_{p^\infty}}_1,x^{{1}/_{p^\infty}}_2,x^{{1}/_{p^\infty}}_3\rangle$ such that $f(\alpha)$ is $(q,s)$-adapted. Then $f$ is surjective.  
		\end{lemma}

	\begin{proof}
		Given $\beta\in C(y)$ we construct an element $\alpha_\infty\in C\langle x^{{1}/_{p^\infty}}_1,x^{{1}/_{p^\infty}}_2,x^{{1}/_{p^\infty}}_3\rangle$ mapping to $\beta$.  		
		Replacing $\beta$ by $\varpi^k \beta$ for some $k$, we may assume $|\beta|_y\leq s$. We can write 
		\begin{equation}
		\beta=\sum \limits_{q\in \bbZ[\frac{1}{p}]} b_{q} x^{q}. 
		\end{equation}

		We construct recursively elements $\alpha_m\in R_3$ and $\beta_m\in C(y)$ for which we prove inductively that: 
		\begin{itemize}
			\item $|\beta_m|_y\leq |\varpi^m|\cdot s$.
			\item  $|\alpha_{m+1}-\alpha_m|_y\leq |\varpi^m|$.
			\item $f(\alpha_m)=\beta -\beta_m$.
		\end{itemize}
		In particular, the sequence $\alpha_m$ is a Cauchy sequence in $C\langle x^{{1}/_{p^\infty}}_1,x^{{1}/_{p^\infty}}_2,x^{{1}/_{p^\infty}}_3\rangle$ and letting $\alpha_\infty$ be the limit of the $\alpha_m$ exhibits an element with $f(\alpha_\infty)=\beta-\beta_\infty=\beta$. 
		One should think of this recursion as a division algorithm and it is done as follows: 

			 We let $\beta_0=\beta$ and $\alpha_0=0$, and by our hypothesis above $|\beta_0|_y\leq |\varpi^0|\cdot s$.
			 If $\beta_m=\sum \limits_{q\in \bbZ[\frac{1}{p}]} b_{q,m} x^{q}$, we let $\{\beta_m\}=\sum \limits_{q\in S_m\subseteq \bbZ[\frac{1}{p}]} b_{q,m} x^{q}$ where: 
			 \[S_m=\{q\in \bbZ[\frac{1}{p}]\mid |\varpi^{m+1}|\cdot  s \leq |b_{q,m} x^{q}|_y \leq |\varpi^{m}|\cdot s\}.\]
			 Since $\beta_m\in C(y)$, $S_m$ is finite.
			 For each $q\in S_m$ we choose an element $e_{q,m}\in O_C\langle x^{{1}/_{p^\infty}}_1,x^{{1}/_{p^\infty}}_2,x^{{1}/_{p^\infty}}_3\rangle$ mapping to an element $f(e_{q,m})$ that is $(q,s)$-adapted. Moreover, we let $c_{q,m}x^q\in C(y)$ denote the leading monomial of $f(e_{q,m})$.   
				
			 Now, by construction for all $q\in S_m$ $|\varpi^{m+1}|\cdot s \leq |b_{q,m} x^{q}|_y\leq |\varpi^m|\cdot s$ and since $f(e_{q,m})$ is $(q,s)$-adapted $|\frac{b_{q,m}x^q}{\varpi^m}|_y< |c_{q,m}x^q|_y$. 
				We may find an element $d_{q,m}\in O_C$ with $|d_{q,m}|< 1$ and such that $b_{q,m}=c_{q,m} d_{q,m} \varpi^m$. 
				We let $e_m=\sum \limits_{q\in S_m} d_{q,m} \varpi^m e_{q,m}$. 

				Notice that by construction, and since $f(e_{q,m})$ is $(q,s)$-adapted, $|f(e_m)-\{\beta_m\}|_y\leq s\cdot  |\varpi^{m+1}|$.  
			 We let $\alpha_{m+1}=\alpha_m+e_m$ and we let $\beta_{m+1}=\beta_m - f(e_m)$. 
			 Note that $|\beta_{m+1}|_y$ is bounded by the maximum of $|\beta_m-\{\beta_m\}|_y$ and $|\{\beta_m\}-f(e_m)|_y$. In particular, $|\beta_{m+1}|_y\leq |\varpi^{m+1}|\cdot s$.
			 It is clear from the definition of $\alpha_{m+1}$ that $|\alpha_{m+1}-\alpha_m|_y\leq|\varpi^m|$ since $|e_m|\leq |\varpi^m|$ holds by construction.
			 Now, $f(\alpha_{m+1})=f(\alpha_m)+f(e_m)=\beta-\beta_m+f(e_m)=\beta-\beta_{m+1}$, by the definition of $\beta_{m+1}$.
	\end{proof}

	We now pass to the second part of the proof. Namely, we construct a map $R_3\to C(y)$ satisfying the hypothesis of \Cref{lema-adapted}. 
	\begin{proposition}
		\label{lema-construction}	
		For any $0<s<1$, there is a continuous map $f:R_3\to C(y)$ satisfying that for any $q\in \bbZ[\frac{1}{p}]$ there is an element $a_q\in  O_C\langle x^{{1}/_{p^\infty}}_1,x^{{1}/_{p^\infty}}_2,x^{{1}/_{p^\infty}}_3\rangle$ such that $f(a_q)$ is $(q,s)$-adapted. 
	\end{proposition}
	\begin{proof}
		We fix $0<s<1$ and construct the map. Since we are working in characteristic $p$ it suffices to specify the images of $x_1$, $x_2$ and $x_3$.
	We let $x_1 \mapsto x$ and we let $x_2 \mapsto c x^{-1}$ for some element $c\in O_C$ so that $s<|c x^{-1}|_y<1$. 
	The construction of the image of $x_3$ is more elaborate.

	Choose a bijection $\omega:\bbN \to \bbZ[\frac{1}{p}]$, we use $\omega$ to well-order the elements of $\bbZ[\frac{1}{p}]$. 
	We let $x_3\mapsto \alpha$ with $\alpha\in C(y)$ constructed as a sum $\sum \limits_{m=1}^\infty \alpha_m$ given by:
	\begin{equation}
		\alpha= \sum \limits_{m=1}^\infty \alpha_m= \sum \limits_{m=1}^\infty [e_m x^{\omega(m)}]^{p^{b_m}}
	\end{equation}
Here $\alpha_m=[e_m x^{\omega(m)}]^{p^{b_m}}$, where we fix $e_m\in C$ so that $s<|e_m x^{\omega(m)}|_y< 1$ and then we construct $b_m\in \bbN$ recursively.

	We let $b_1=0$ and for $m>1$ we choose $b_m$ depending on $\{b_1,\dots,b_{m-1}\}$ and sufficiently large so that $b_m$ satisfies the following constraints: 
	\begin{enumerate}
		\item We require that $b_m$ be large enough so that $|\alpha_m|_y<|\varpi|^m$.
		\item Choose $\epsilon_m\in O_C\setminus \{0\}$, with the property that for all $1\leq j< m$, \label{constassumption2}
			\begin{align}
				|\epsilon_m \alpha_j|_y\leq |x^{\omega(j)\cdot p^{b_j}}|_y & &  when & & 0<\omega(j) \\
				|\epsilon_m \alpha_j|_y\leq |(c x^{-1})^{-\omega(j)p^{b_j}}|_y & & when & & \omega(j)<0. 
			\end{align}
			With $\epsilon_m$ chosen in this way we require that $b_m$ be large enough so that $s<|\epsilon_m^{p^{-b_m}} e_m x^{\omega(m)}|_y<1$.
		\item For all $1\leq j<m$, we require $b_m-b_j$ to be large enough so that $|(e_{m} x^{\omega(m)})^{p^{b_m-b_j}}|_y< |\varpi|\cdot s$. \label{constassumption3}
	\end{enumerate}
The first condition ensures that $\alpha$ is a well defined element in $C(y)$. 
We will exploit the second and third condition to construct for all $q\in \bbZ[\frac{1}{p}]$ a $(q,s)$-adapted element in the image of $O_C\langle x^{{1}/_{p^\infty}}_1,x^{{1}/_{p^\infty}}_2,x^{{1}/_{p^\infty}}_3\rangle$. 
 
Let $W_{\omega(j)}=x_1^{\omega(j)p^{b_{j}}}$ when $\omega(j)\geq 0$ and $W_{\omega(j)}=x_2^{-\omega(j)\cdot p^{b_{j}}}$ if $\omega(j)<0$. 
By the requirement in \ref{constassumption2}, for $1\leq j\leq m-1$ the $j$th-term of $\epsilon_m\cdot f(x_3)$ is divisible by $f(W_{\omega(j)})$, let $d_{\omega(j)}\in O_C$ denote $\frac{\epsilon_m\cdot \alpha_j}{f(W_{\omega(j)})}$. 
Let 
\begin{equation}
	a_{\omega(m)}^{p^{b_m}}=  \epsilon_m x_3 - \sum \limits_{i=1}^{m-1} d_{\omega(i)} W_{\omega(i)},
\end{equation}
we claim that $a_{\omega(m)}$ which by construction lies in $O_C\langle x^{{1}/_{p^\infty}}_1,x^{{1}/_{p^\infty}}_2,x^{{1}/_{p^\infty}}_3\rangle$, satisfies that $f(a_{\omega(m)})$ is $(\omega(m),s)$-adapted.
Indeed, $f(a_{\omega(m)})^{p^{b_m}}=\epsilon_m\cdot \sum \limits_{j=m}^\infty \alpha_m$, and the first term in the sum expression of $f(a_{\omega(m)})$ is $\epsilon_m^{p^{-b_m}}\cdot e_{m} x^{\omega(m)}$, with value $s<|\epsilon_m^{p^{-b_m}}\cdot e_{m} x^{\omega(m)}|_y<1$ by \cref{constassumption2}. 
The next terms have the form $\epsilon_m^{p^{-b_m}}\cdot (e_{k}x^{\omega(k)})^{p^{b_{k}-b_m}}$ with value $|\epsilon_m^{p^{-b_m}}\cdot (e_{k}x^{\omega(k)})^{p^{b_{k}-b_m}}|_y<|\varpi|\cdot s$ by \cref{constassumption3}. 
This finishes the proof that $f(a_{\omega(m)})$ is $(\omega(m),s)$-adapted.
	\end{proof}

	\end{proof}

	\begin{proposition}
		\label{casen=2}
		If $|C^\times|\neq \bbR^+$, then $(C,R_2)$ is not a Nullstellensatz pair.
	\end{proposition}
	\begin{proof}
		Let $I_c\subseteq R_3$ denote the ideal generated by $x_1\cdot x_2-c$, and let $S_c=R_3/I_c$. 
		By \Cref{lema-adapted} and the explicit construction of \Cref{lema-construction} we know that for $c$ chosen so that $s<|c\cdot x^{-1}|_y<1$ the pair $(C,S_c)$ is not a Nullstellensatz pair since there is a surjection $S_c\to C(y)$.
		We observe that $\Spa S_c=U\times \bbB_C^1$ for $U\subseteq \bbG_{m,C}$ an open subset of the multiplicative group.
		Indeed, the variable $x_3$ does not interact with the other variables and the variety cut by the equation $x_1x_2=c$ can be parametrized by the map $t\mapsto (t,c\cdot t^{-1})$, $U$ corresponds to the locus where $|t|\leq 1$ and $|c\cdot t^{-1}|\leq 1$.

		We can now regard $\Spa S_c=U\times \bbB^1$ as an open subset of $\bbB^2_C$ and by \cite[Corollary D.5]{Zav21} we may find a finite \'etale morphism $\Spa S_c\to \bbB^2_C$.
		This gives a map $R_2\to C(y)$, and we let $K$ denote the image of this map. 
		Since $S_c\to C(y)$ is surjective $C(y)$ is a finite $K$-module. 
		This implies that $K$ is a field.
		We have constructed a surjective map $R_2\to K$ and $K\neq C$ since $C(y)$ is a finite extension of $K$ but not of $C$.
	\end{proof}

	In the proof of \Cref{mainnotype3point} we used the following lemma. 

	\begin{lemma}
		\label{ExplicitTypeIII}
		Let $F$ be a complete non-Archimedean field and let $y\in \bbB^1_F$ be a Type III point corresponding to the ball of radius $r$ and center $0$ for some $r<1$ and $r\in \bbR^+\setminus |F^\times|^\bbQ$. 
		The following hold:
		\begin{enumerate}
			\item Let $F(y)$ denote the residue field at $y$. Then $F(y)$ is isomorphic to the ring of $r$-convergent power series expressions: 
				\[\sum_{i\in \bbZ} a_i x^i\]
				where $a_i\in F$ and $\lim \limits_{|i|\to \infty} |a_i|\cdot r^i=0$.
			\item Suppose that $F$ is perfect of characteristic $p$ and let $F(y)^{\on{perf}}$ be the perfected residue field at $y$. Then $F(y)^{\on{perf}}$ is isomorphic to the ring of perfected $r$-convergent power series expressions: 
				\[\sum_{i\in \bbZ[\frac{1}{p}]} a_i x^i\]
				where $a_i\in F$ and for all $\varepsilon>0$ the set $\{i\in \bbZ[\frac{1}{p}]\mid |a_i|\cdot r^i\geq \varepsilon\}$ is finite. 
		\end{enumerate}
	\end{lemma}
		\begin{proof}
			Let $F\langle x\rangle_r$ and $F\langle x^{\frac{1}{p^\infty}}\rangle_r$ denote the rings of $r$-convergent power series and perfected $r$-convergent power series respectively.
			Observe that since $r\notin |F^\times|^\bbQ$ the rings $F\langle x\rangle_r$ and $F\langle x^{\frac{1}{p^\infty}}\rangle_r$ are fields. 
			Indeed, let $\Gamma\in \{\bbZ,\bbZ[\frac{1}{p}]\}$ if $f=\sum_{i\in \Gamma} a_i x^i$ the set $\{|a_i|r^i\}_{i\in \Gamma}$ attains a unique maximum, say $|a_{N}|\cdot r^N$. 
			Then $n=(a_N^{-1} x^{-N} f-1)$ is a topologically nilpotent element which implies that $a_N^{-1}x^{-N}f=1+n$ is invertible.

			Since $r<1$ we have tautological inclusions $T_1=F\langle x\rangle_1\to F\langle x\rangle_r$ and $R_1=F\langle x^{\frac{1}{p^\infty}}\rangle_1\to F\langle x^{\frac{1}{p^\infty}}\rangle_r$ determined by $x\mapsto x$. 
			Moreover, the valuation induced by the valued fields $F\langle x\rangle_r$ and $F\langle x^{\frac{1}{p^\infty}}\rangle_r$ is precisely the one induced by $y$. 
			Thus we get factorizations $T_1\to F(y)\to F\langle x\rangle_r$ and $T_1\to F(y)^{\on{perf}}\to F\langle x^{\frac{1}{p^\infty}}\rangle_r$.
			It suffices to show that $F(y)\subseteq F\langle x\rangle_r$ and $F(y)^{\on{perf}}\to F\langle x^{\frac{1}{p^\infty}}\rangle_r$ are dense, but this is the case since $T_1$ and $R_1$ are already dense in $F\langle x\rangle_r$ and $F\langle x^{\frac{1}{p^\infty}}\rangle_r$ respectively. 
		\end{proof}

	\bibliography{biblio.bib}
	\bibliographystyle{alpha}
	
\end{document}